\documentclass[11pt, notitlepage]{article}
\usepackage{amssymb,amsmath,comment}
\usepackage{amsthm}
\catcode`\@=11 \@addtoreset{equation}{section}
\def\thesection{\arabic{section}}

\def\theequation{\thesection.\arabic{equation}}
\catcode`\@=12
\usepackage{colortbl}
\usepackage{hyperref}
\usepackage[mathscr]{eucal}
\usepackage{epsf}
\usepackage{esint}
\usepackage{a4wide}

\newcommand{\noi} {\noindent}
\newcommand{\na} {\nabla}

\newcommand{\mb} {\mathbb}

\usepackage[all]{xy}
\catcode`\@=11
\def\theequation{\@arabic{\c@section}.\@arabic{\c@equation}}
\catcode`\@=12

\newtheorem{Theorem}{Theorem}[section]
\newtheorem{Lemma}[Theorem]{Lemma}

\newtheorem{Corollary}[Theorem]{Corollary}
\newtheorem{Remark}[Theorem]{Remark}
\newtheorem{Definition}[Theorem]{Definition}

\begin{document}

{\vspace{0.01in}}

\title{Mixed local and nonlocal weighted singular quasilinear elliptic problem and its associated Sobolev-type inequality}

\author{Prashanta Garain\footnote{Department of Mathematical Sciences\\
Indian Institute of Science Education and Research Berhampur, Permanent Campus, At/Po:-Laudigam, Dist.-Ganjam,
Odisha, India-760003,\,\textbf{Email}: pgarain92@gmail.com}}

\maketitle

\begin{abstract}\noindent
We consider a class of mixed anisotropic and nonlocal singular quasilinear elliptic problem associated with Muckenhoupt weights. The presence of the singular nonlinearity, which blows up near the origin along with its interaction with anisotropy, weighted degeneracy, and nonlocal diffusion creates significant analytical challenges. We employ monotone approximation, weighted Sobolev embeddings, compactness, and variational methods to establish the existence and uniqueness of weak solutions under suitable assumptions on the datum. Further, we characterize the best constant in an associated weighted mixed anisotropic and nonlocal Sobolev-type inequality, prove that it is attained, and show that the normalized weak solution is the unique extremal. These results are new even in the mixed weighted Laplace case \(p=2\).
\end{abstract}

\maketitle

\noi {Keywords: Mixed anisotropic and nonlocal weighted quasilinear equation, singular nonlinearity, weighted anisotropic and nonlocal Sobolev-type inequality, extremal.}

\noi{\textit{2020 Mathematics Subject Classification: 35J92, 35R11, 35J70, 35J75, 35A23}

\bigskip

\tableofcontents

\section{Introduction and main results}
\subsection{Introduction}
Singular elliptic equations have attracted considerable attention in recent years because of their rich mathematical structure and their broad range of applications in physics, geometry, nonlinear elasticity, and the theory of non-Newtonian fluids, see \cite{BPugh}. Here, singularity refers to the fact that the nonlinearity blows up near the origin, making the analysis highly challenging. As a result, establishing the existence, uniqueness, regularity, and other qualitative properties of solutions requires delicate analytical techniques.

Mixed local--nonlocal operators combine classical diffusion with fractional diffusion, thereby providing a more realistic mathematical framework for modeling a variety of physical phenomena \cite{DVap}. While significant progress has been made in the study of purely local and purely nonlocal elliptic problems, the corresponding theory for mixed local--nonlocal operators remains far from complete, particularly in the presence of singular nonlinearities.

In recent years, substantial research has also been devoted to degenerate elliptic problems, for which weighted Sobolev spaces provide the natural functional framework, see for example \cite{Heinonen}. Degenerate elliptic equations involving Muckenhoupt weights arise naturally in harmonic analysis and degenerate diffusion processes. Unlike the uniformly elliptic setting, Muckenhoupt weights may either vanish or become unbounded, leading to additional analytical difficulties and significantly increasing the complexity of the problem.

Motivated by these developments, in this paper we investigate the following weighted mixed anisotropic and nonlocal singular problem
\begin{equation}\label{meqn}
\mathcal{M}_{\alpha}u=f(x)u^{-\delta}
\quad\text{in }\Omega,\qquad
u>0\quad\text{in }\Omega,\qquad
u=0\quad\text{in }\mathbb{R}^{N}\setminus\Omega,
\end{equation}
where $\Omega\subset\mathbb{R}^{N}$, $N\ge2$, is a bounded Lipschitz domain, $0<\delta<1$, and $f\in L^{m}(\Omega)\setminus\{0\}$ is a nonnegative function, where the precise assumptions on $m$ will be specified later. Here the operator $\mathcal{M}_{\alpha}$ is defined by
\begin{equation}\label{Mop}
\mathcal{M}_{\alpha}u:=-H_{p,w}u+(-\Delta_{p,w})^{s}u,\quad 1<p<\infty,
\end{equation}
where for $w\in A_p$ (Muckenhoupt weights, see section 2 for more details)
\[
H_{p,w}u
:=
\mathrm{div}\!\left(
w(x)H(\nabla u)^{p-1}\nabla H(\nabla u)
\right)
\]
is the weighted anisotropic $p$-Laplace operator associated with a Finsler-Minkowski norm $H$ (refer to section 2 for more details). For $0<s<1<p<\infty$ and $w\in A_p$, the weighted fractional $p$-Laplace operator is defined by (see \cite{Ok24})
\[
(-\Delta_{p,w})^{s}u
:=
\mathrm{P.V.}
\int_{\mathbb{R}^{N}}
\frac{|u(x)-u(y)|^{p-2}(u(x)-u(y))}
{|x-y|^{sp}}
K(x,y)\,dy
\]
is the weighted fractional $p$-Laplace operator, where P.V. denotes the principal value, and for almost every $x,y\in\mb{R}^N$, we have
\[
K(x,y)
:=
\frac{w(x)w(y)}{w(B_{x,y})},
\]
and
\[
B_{x,y}
:=
B_{\frac{|x-y|}{2}}
\left(
\frac{x+y}{2}
\right).
\]
Therefore, we observe that the operator $\mathcal{M}_{\alpha}$ has three different type of features simultaneously, namely, anisotropy, nonlocality, and weighted degeneracy.

In the purely local setting, singular problems of the form
\begin{equation}\label{plap}
-\Delta_p u=g(x)u^{-\delta}\quad\text{in }\Omega,
\qquad
u=0\quad\text{on }\partial\Omega,
\end{equation}
have been extensively studied under various assumptions on the singular datum $g$ and the domain $\Omega$. In the semilinear case ($p=2$), the existence of a unique positive classical solution was established in the celebrated work \cite{CRT}. Later, it was observed in \cite{Lz} that such solutions are weak solutions for a suitable range of the singular exponent $\delta$. This restriction was later removed in \cite{Boc-Or}, where existence and regularity results were obtained for all $\delta>0$. These results were further extended to the quasilinear setting in \cite{Canino,DeCave}. We also refer to \cite{Garainmn,GMmed} and the references therein for the study of corresponding weighted $p$-Laplace problems.

An important feature of problem \eqref{plap} is its intimate connection with Sobolev-type inequalities. Indeed, it was shown in \cite{Annali} that weak solutions of \eqref{plap} characterize the extremals of an associated Sobolev-type inequality. This relationship was subsequently extended to the weighted setting in \cite{BGmm22,BGdie} and Carnot group in \cite{GUaamp} respectively.

In the purely nonlocal setting, the singular problem
\begin{equation}\label{fplap}
(-\Delta_p)^s u=g(x)u^{-\delta}
\quad\text{in }\Omega,
\qquad
u=0
\quad\text{in }\mathbb{R}^{N}\setminus\Omega,
\end{equation}
has also received considerable attention in recent years. Existence and uniqueness results for arbitrary $\delta>0$ were established in \cite{Caninoetal}; see also \cite{Garaincpaa,GSanona} for further developments. Moreover, it was proved in \cite{Nmn} that weak solutions of \eqref{fplap} are closely related to the extremals of a nonlocal Sobolev-type inequality. Recently such results have been obtained in the Heisenberg group in \cite{GarainHein}.

We note that problem \eqref{meqn} extends the unweighted mixed local--nonlocal singular equation
\begin{equation}\label{meqnp1}
-\Delta_pu+\alpha\,(-\Delta_p)^su
=
f(x)u^{-\delta}
\quad\text{in }\Omega,
\qquad
u>0
\quad\text{in }\Omega,
\qquad
u=0
\quad\text{in }\mathbb{R}^{N}\setminus\Omega.
\end{equation}
Problem \eqref{meqnp1} and its variants have been the subject of extensive investigation in recent years. In the semilinear case ($p=2$), existence and qualitative properties of weak solutions were established in \cite{ARad,HH} and the references therein. Subsequently, perturbed mixed local--nonlocal singular problems were studied in \cite{AGmz,Ghoshjga,Vecchi,BGjga,BGccm,Garainjga}, where existence and regularity results were obtained. In the quasilinear setting ($1<p<\infty$), existence and qualitative properties of singular mixed local--nonlocal problems were investigated in \cite{BDp,Biroud,Gjms,GKK,GU21}, while the corresponding perturbed problems were considered in \cite{BDpur,Dna,Gpur}; see also the references therein.

A remarkable feature of problem \eqref{meqnp1} is its close connection with mixed local--nonlocal Sobolev-type inequalities. Indeed, it was shown in \cite{GU21} that weak solutions of \eqref{meqnp1} characterize the extremals of an associated mixed Sobolev-type inequality. Moreover, in the Heisenberg group such connections have been proved in \cite{Gop2}. More recently, this connection has been extended to the anisotropic mixed local--nonlocal framework in \cite{Gjms}.

Despite these developments, to the best of our knowledge, no existence and uniqueness theory is available for singular mixed local and nonlocal equations in the presence of Muckenhoupt weights, even in the mixed local and nonlocal Laplace case $p=2$. To the best of our knowledge, the only existing work on mixed local and nonlocal weighted problems is the recent paper \cite{BG26}, which is devoted to the regularity theory of problems involving nonsingular nonlinearities. Likewise, the associated weighted mixed Sobolev-type inequalities, the characterization of their best constants, and the uniqueness of extremals have not yet been investigated. The principal objective of the present paper is to fill this gap.

The analysis of problem \eqref{meqn} presents several new mathematical challenges. First, the singular reaction term $u^{-\delta}$ becomes unbounded near the zero set of the solution, preventing the direct application of standard variational arguments. Second, since the Muckenhoupt weight may vanish or become unbounded, the weighted anisotropic operator is generally neither uniformly elliptic nor uniformly coercive. Third, the fractional operator introduces long-range interactions, so the local and nonlocal energies must be handled simultaneously. Finally, the interaction between anisotropy, weighted degeneracy, nonlocal diffusion, and the singular nonlinearity creates analytical difficulties that are absent in the uniformly elliptic or unweighted settings.

To overcome these difficulties, we mainly follow the approaches from \cite{Boc-Or} and \cite{Nmn}. To be more precise, our approach combines monotone approximation with nonlinear monotone operator theory in weighted Sobolev spaces. A key ingredient is the development of weighted mixed energy estimates allowing the simultaneous treatment of the anisotropic local operator and the weighted fractional operator. These estimates provide the compactness needed to pass to the limit in the singular approximation scheme and also yield the variational characterization of extremals.

Our first main result (Theorem \ref{thm1}) establishes the existence and uniqueness of weak solutions in the weighted Sobolev space under suitable integrability assumptions on the datum $f$. Our second main contribution (Theorem \ref{thm5}) concerns the variational structure associated with problem \eqref{meqn}. We identify the optimal constant in a weighted mixed anisotropic and nonlocal Sobolev-type inequality through the unique weak solution of the singular equation. Furthermore, we prove that this optimal constant is attained, establish the simplicity of all extremals, and show that the normalized weak solution is the unique extremal of the associated constrained minimization problem. These results reveal a close relationship between the singular equation and the corresponding weighted Sobolev-type inequality.

As far as we are aware, all the main results obtained in this paper are new. Even for the linear case $p=2$, neither the well-posedness theory nor the extremal characterization of weighted mixed anisotropic and nonlocal Sobolev-type inequalities appears to be available in the literature. Consequently, this paper provides the first unified treatment of singular equations involving the simultaneous presence of a singular reaction term, anisotropic diffusion, weighted degeneracy, and nonlocal interactions. Besides their intrinsic mathematical interest, the techniques developed here furnish a robust framework for studying a broader class of weighted mixed local--nonlocal problems, including equations with variable nonlinearities, measure data, and more general singular or nonsingular reaction terms.

\subsection*{Organization of the paper
}
In Section 3, we discuss the functional setting and known results. In Section 4, we establish some preliminary results and finally, in Section 5, we prove our main results.

\subsection{Main results}
First we have the following existence and uniqueness result.
\begin{Theorem}\label{thm1}
Assume that $\alpha>0$ and $0<s<1<p<\infty$ with $w\in A_p^{t}$, where $A_p^t$ is defined in \eqref{w}. Let $\delta\in(0,1)$ and \(f\in L^{m}(\Omega)\setminus\{0\}\) be a nonnegative function, where \(m\) is given by \eqref{m}. Then problem~\eqref{meqn} admits a unique weak solution
$
u_\delta\in W_{0}^{1,p}(\Omega,w).
$
\end{Theorem}
Now, we state the variational characterization result.
\begin{Theorem}\label{thm5}
Assume that $\alpha>0$ and $0<s<1<p<\infty$ with $w\in A_p^{t}$, where $A_p^t$ is defined in \eqref{w}. Let \(f\in L^{m}(\Omega)\setminus\{0\}\) be a nonnegative function, where \(m\) is given by \eqref{m}, and let
$
u_\delta\in W_{0}^{1,p}(\Omega,w)
$
be the unique weak solution of problem~\eqref{meqn} obtained in Theorem~\ref{thm1}. Then the following assertions hold:

\begin{enumerate}
\item[(a)] \textbf{(Extremal property)}
\begin{align*}
\Theta(\Omega)
&:=\inf_{v\in W^{1,p}_{0}(\Omega,w)\setminus\{0\}}
\left\{
\int_{\Omega} H(\na v)^{p}\,w\,dx
+\alpha
\int_{\mathbb{R}^{N}}\int_{\mathbb{R}^{N}}
|v(x)-v(y)|^p
\,d\mu
:\;
\int_{\Omega}|v|^{1-\delta}f\,dx=1
\right\}\\
&=\left(
\int_{\Omega}
H(\na u_\delta)^{p}\,w\,dx
+
\alpha\int_{\mathbb{R}^{N}}\int_{\mathbb{R}^{N}}
|u_\delta(x)-u_\delta(y)|^{p}\,d\mu
\right)^{\frac{1-\delta-p}{1-\delta}},
\end{align*}
where
$$
d\mu:=\frac{K(x,y)}{|x-y|^{sp}}\text{ as defined in \eqref{dmu}}.
$$

\item[(b)] \textbf{(Weighted anisotropic and nonlocal Sobolev-type inequality)} For every \(v\in W_{0}^{1,p}(\Omega,w)\), the weighted anisotropic and nonlocal Sobolev-type inequality
\begin{equation}\label{inequality2}
C
\left(
\int_{\Omega}
|v|^{1-\delta}f\,dx
\right)^{\frac{p}{1-\delta}}
\le
\int_{\Omega}
H(\na v)^{p}\,w\,dx
+
\alpha\int_{\mathbb{R}^{N}}\int_{\mathbb{R}^{N}}
|v(x)-v(y)|^{p}\,d\mu,
\end{equation}
holds, if and only if
\[
C\le\Theta(\Omega).
\]

\item[(c)] \textbf{(Simplicity of extremals)} Suppose that \(v\in W_{0}^{1,p}(\Omega,w)\) satisfies
\begin{equation}\label{sim}
\Theta(\Omega)
\left(
\int_{\Omega}
|v|^{1-\delta}f\,dx
\right)^{\frac{p}{1-\delta}}
=
\int_{\Omega}
H(\na v)^{p}\,w\,dx
+
\alpha\int_{\mathbb{R}^{N}}\int_{\mathbb{R}^{N}}
|v(x)-v(y)|^{p}\,d\mu.
\end{equation}
Then there exists a constant \(k\in\mathbb{R}\) such that
$
v=ku_\delta.
$
\end{enumerate}
\end{Theorem}
\begin{Corollary}\label{ansiowgtrmk}
We define
\[
S_{\delta}
:=
\left\{
v\in W_{0}^{1,p}(\Omega,w):
\int_{\Omega}|v|^{1-\delta}f\,dx=1
\right\},
\]
and we set
\[
\zeta_{\delta}
:=
\left(
\int_{\Omega}u_{\delta}^{\,1-\delta}f\,dx
\right)^{-\frac{1}{1-\delta}}.
\]
Then
\[
V_{\delta}:=\zeta_{\delta}u_{\delta}\in S_{\delta}.
\]
Furthermore, Theorem~\ref{thm5} implies that
\[
\Theta(\Omega)
=
\int_{\Omega}
H(\na V_{\delta})^{p}\,w\,dx
+
\int_{\mathbb{R}^{N}}\int_{\mathbb{R}^{N}}
|V_{\delta}(x)-V_{\delta}(y)|^{p}\,d\mu.
\]
In particular, \(V_{\delta}\) is an extremal for the weighted Sobolev-type inequality \eqref{inequality2}. Moreover, \(V_{\delta}\) is the unique weak solution of
\[
\mathcal{M}_{\alpha}V_{\delta}
=
\Theta(\Omega)\,f\,V_{\delta}^{-\delta}
\quad\text{in }\Omega,
\qquad
V_{\delta}>0
\quad\text{in }\Omega,
\]
with
\[
V_{\delta}=0
\quad\text{in }\mathbb{R}^{N}\setminus\Omega.
\]
\end{Corollary}

\section{Functional setting and known results}
\subsection{Functional setting}
In this section, we recall some standard results concerning Muckenhoupt weights and the associated weighted Sobolev spaces that will be used throughout the paper.

We say that $w$ is a weight function if $w\in L_{\mathrm{loc}}^{1}(\mathbb{R}^{N})$ and satisfy
$
0<w(x)<\infty \quad \text{for a.e. } x\in\mathbb{R}^{N}.
$
\begin{Definition}[Muckenhoupt weight]
Let $1<p<\infty$. Then, we say that a weight function $w$ belongs to the Muckenhoupt class $A_p$, if there exists a constant $C>0$ depending only on $p$ and $w$ such that
\[
\left(\frac{1}{|B|}\int_B w(x)\,dx\right)
\left(\frac{1}{|B|}\int_B w(x)^{-\frac{1}{p-1}}\,dx\right)^{p-1}
\le C
\]
for every ball $B\subset\mathbb{R}^{N}$.
The smallest such constant is called the \emph{$A_p$-constant} of $w$ and is denoted by $[w]_{A_p}$.

A fundamental example is the power weight
$
w(x)=|x|^\alpha,
$
which belongs to $A_p$ if and only if
$
-N<\alpha<N(p-1).
$
\end{Definition}

Next, we recall the definition of the weighted Sobolev space associated with a Muckenhoupt weight.

\begin{Definition}[Weighted Sobolev space]
Let $E\subset\mathbb{R}^{N}$ be an open set and let $w\in A_p$ for some $1<p<\infty$. The weighted Sobolev space $W^{1,p}(E,w)$ is defined by
\[
W^{1,p}(E,w)
:=
\left\{
u\in L^{p}(E,w): |\nabla u|\in L^{p}(\Omega,w)
\right\},
\]
where
\[
L^{p}(E,w)
:=
\left\{
u:E\to\mathbb{R}\ \text{measurable}:
\int_E |u|^p w\,dx<\infty
\right\}.
\]
The space $W^{1,p}(E,w)$ is endowed with the norm
\[
\|u\|_{W^{1,p}(E,w)}
:=
\left(
\int_E |u|^p\,w\,dx
\right)^{\frac1p}
+
\left(
\int_E H(\nabla u)^p\,w\,dx
\right)^{\frac1p}.
\]
\end{Definition}
Let $\alpha>0$, and $\Omega\subset\mb{R}^N$ with $N\geq 2$ be a bounded Lipschitz domain, then we define the space
$
W_0^{1,p}(\Omega,w)
$
as the closure of $C_c^\infty(\Omega)$ with respect to the norm
\begin{equation}\label{norm}
\|u\|_{W_0^{1,p}(\Omega,w)}:=\left(\int_{\Omega}H(\na u)^p\,w\,dx+\alpha[u]^p_{w,s,p}\right)^\frac{1}{p},
\end{equation}
where
$$
[u]_{w,s,p}^p:=\int_{\mb{R}^N}\int_{\mb{R}^N}|u(x)-u(y)|^p\,d\mu,
$$
and we recall that
\begin{equation}\label{dmu}
d\mu:=\frac{K(x,y)}{|x-y|^{sp}}\,dx\,dy,
\end{equation}
with
$$
K(x,y):=\frac{w(x)w(y)}{w(B_{x,y})},\text{ for almost every }x,y\in\mb{R}^N,
$$
and
$$
B_{x,y}:=B_{\frac{|x-y|}{2}}\Big(\frac{x+y}{2}\Big),
$$
is the open ball of radius $|x-y|/2$ with center at $(x+y)/2$. We observe that the space $W_0^{1,p}(\Omega,w)$ is a uniformly convex Banach space. For further properties of weighted Sobolev spaces, we refer the reader to \cite{Drabek, Heinonen} and the references therein. Further, for more details on the fractional Sobolev and fractional weighted Sobolev spaces, we refer to \cite{Ok24, Hitchhiker'sguide}.

Here $H:\mb{R}^N\to[0,\infty)$ is  a Finsler-Minkowski norm, that is \(H:\mathbb{R}^N\to[0,\infty)\) in
\(C^1(\mathbb{R}^N\setminus\{0\})\) is a strictly convex function satisfying
\[
H(\xi)=0 \iff \xi=0,\qquad
H(t\xi)=|t|H(\xi), \quad \forall\,\xi\in\mathbb{R}^N,\; t\in\mathbb{R},
\]
and there exist positive constants \(c_1,c_2\) such that
\[
c_1|\xi|\leq H(\xi)\leq c_2|\xi|,
\qquad \forall\,\xi\in\mathbb{R}^N.
\]
Throughout the paper, the notation
$
\nabla H(\nabla u)
:=
\left.\nabla_{\xi}H(\xi)\right|_{\xi=\nabla u(x)}
$
is used, where \(\nabla_{\xi}H(\xi)\) denotes the gradient of \(H\) with respect to its argument \(\xi\), while \(\nabla u(x)\) is the gradient of \(u\) with respect to the spatial variable \(x\).

For any $1<p<\infty$, it follows from the proof of \cite[Lemma 5.9]{Heinonen} that
\begin{equation}\label{alg2}
\left\langle
H(x)^{p-1}\nabla H(x)
-
H(y)^{p-1}\nabla H(y),
\,x-y
\right\rangle
>0,
\qquad
\forall\, x,y\in\mathbb{R}^N,\; x\neq y.
\end{equation}

The following result follows from \cite[Proposition 2.1]{FW} and \cite[Proposition 1.2]{Xiathesis}.

\begin{Lemma}\label{prop}
Let $H:\mb{R}^N\to[0,\infty)$ be a Finsler-Minkowski norm. Then the following properties hold.

\begin{enumerate}
\item[(i)] For every \(x\in\mathbb{R}^N\setminus\{0\}\),
\begin{equation}\label{eq:H-euler}
x\,\nabla H(x)=H(x).
\end{equation}

\item[(ii)] For every \(x\in\mathbb{R}^N\setminus\{0\}\) and \(t\in\mathbb{R}\setminus\{0\}\),
\begin{equation}\label{eq:H-gradient-homogeneity}
\nabla H(tx)=\operatorname{sign}(t)\,\nabla H(x).
\end{equation}

\item[(iii)] There exists a positive constant \(C\) such that
\begin{equation}\label{eq:H-gradient-bound}
|\nabla H(x)|\le C,
\qquad \forall\,x\in\mathbb{R}^N\setminus\{0\}.
\end{equation}
\end{enumerate}
\end{Lemma}

To gain further insight into the operator $H_{p,w}$, we present some important examples of Finsler--Minkowski norms. For additional background on anisotropic \(p\)-Laplace operators, we refer the reader to \cite{BFKzamp, Xiathesis, MV} and the references therein.

\medskip

\noindent\textbf{Examples.} Let \(x=(x_1,x_2,\ldots,x_N)\in\mathbb{R}^N\).

\begin{enumerate}
\item[(i)] For \(q>1\), define
\begin{equation}\label{ex11}
F_q(x)
:=
\left(\sum_{i=1}^{N}|x_i|^q\right)^{\frac{1}{q}}.
\end{equation}

\item[(ii)] For \(\lambda,\mu>0\), define
\begin{equation}\label{ex2}
F_{\lambda,\mu}(x)
:=
\sqrt{\lambda\sqrt{\sum_{i=1}^{N}x_i^{4}}
+\mu\sum_{i=1}^{N}x_i^{2}}.
\end{equation}
\end{enumerate}

It follows from \cite{MV} that the functions \(F_q\) and \(F_{\lambda,\mu}\), defined by \eqref{ex11} and \eqref{ex2}, respectively, are Finsler--Minkowski norms on \(\mathbb{R}^N\).

\begin{Remark}\label{exrmk1}
Let \(i=1,2\), and let \(\lambda_i,\mu_i>0\) satisfy
\[
\frac{\lambda_1}{\mu_1}\neq\frac{\lambda_2}{\mu_2}.
\]
Then the Finsler--Minkowski norms \(F_{\lambda_1,\mu_1}\) and \(F_{\lambda_2,\mu_2}\) are nonisometric on \(\mathbb{R}^N\); see \cite{MV}.
\end{Remark}

\begin{Remark}\label{exrmk2}
For the choice \(H=F_q\) in \eqref{ex11}, the weighted anisotropic \(p\)-Laplace operator is given by
\[
H_{p,w}u
=
\sum_{i=1}^{N}\frac{\partial}{\partial x_i}
\left(
w(x)
\left(
\sum_{k=1}^{N}
\left|
\frac{\partial u}{\partial x_k}
\right|^{q}
\right)^{\frac{p-q}{q}}
\left|
\frac{\partial u}{\partial x_i}
\right|^{q-2}
\frac{\partial u}{\partial x_i}
\right).
\]
This family contains several classical operators as special cases. Indeed,
\[
H_{p,w}u=
\begin{cases}
\operatorname{div}\!\left(w(x)|\nabla u|^{p-2}\nabla u\right),
& \text{if } q=2,\;1<p<\infty,\\[1.2ex]
\displaystyle
\sum_{i=1}^{n}
\frac{\partial}{\partial x_i}
\left(w(x)|u_i|^{p-2}u_i\right),
& \text{if } q=p\in(1,\infty),
\end{cases}
\]
where \(u_i=\dfrac{\partial u}{\partial x_i}\) for \(i=1,\ldots,N\). Consequently, \(H_{p,w}\) reduces to the weighted \(p\)-Laplace operator when \(q=2\), and to the weighted pseudo \(p\)-Laplace operator when \(q=p\).
\end{Remark}

Now we introduce the notion of weak solution to problem~\eqref{meqn}.

\begin{Definition}[Weak solution]\label{wksoldef}
Let $\alpha>0$ and $0<s<1<p<\infty$ with $w\in A_p$. Suppose that \(f\in L^{1}(\Omega)\setminus\{0\}\) is a nonnegative function. A function
$
u\in W_{0}^{1,p}(\Omega,w)
$
is said to be a \emph{weak solution} of problem~\eqref{meqn} if the following conditions are satisfied:
\begin{enumerate}
    \item for every compact set \(\omega\Subset\Omega\), there exists a constant \(C(\omega)>0\) such that
    \[
    u\ge C(\omega)\quad\text{in }\omega;
    \]
    \item for every \(\varphi\in C_{c}^{1}(\Omega)\),
    \begin{equation}\label{wksoleqn}
    \begin{split}
    &\int_{\Omega}
    H(\na u)^{p-1}\nabla H(\na u)\na\varphi\,w\,dx
    +\alpha
    \int_{\mathbb{R}^{N}}\int_{\mathbb{R}^{N}}
    J_p(u(x)-u(y))
    \bigl(\varphi(x)-\varphi(y)\bigr)\,d\mu\\
&\quad    =
    \int_{\Omega}
    f(x)u^{-\delta}\varphi\,dx.
    \end{split}
    \end{equation}
\end{enumerate}
\end{Definition}

\begin{Remark}\label{Rmwksol}
By Lemma \ref{prop}, Definition \ref{wksoldef} is well stated. By a standard density argument (see \cite[Lemma~5.1]{GU21}), the identity \eqref{wksoleqn} remains valid for every test function \(\varphi\in W_{0}^{1,p}(\Omega,w)\).
\end{Remark}

\subsection*{Notations and assumptions}
Throughout the rest of paper, we adopt the following notation and standing assumptions unless explicitly stated otherwise.

\begin{itemize}

\item \(\Omega\) denotes a bounded Lipschitz domain in \(\mathbb{R}^N\) with $N\geq 2$.

    \item For \(l>1\), we denote by
    $
    l'=\frac{l}{l-1}
    $
    the H\"older conjugate exponent of \(l\).

    \item We assume throughout that
    $
    \alpha>0,\,0<s<1<p<\infty,\,\delta\in(0,1).
    $
    Moreover, we assume that $w\in A_p^{t}$ for $t\in I:=(1/(p-1),\infty)\cap (N/p,\infty)$. Further, we use the notation
$p_t:=\frac{pt}{t+1}\in[1,p),$
along with the notation
$
p_t^{*}:=\frac{Np_t}{N-p_t}
$
whenever \(p_t<N\).

\item We assume that $f\in L^m(\Omega)\setminus\{0\}$ is nonegative, where
\begin{equation}\label{m}
     m:=
     \begin{cases}
         \Big(\frac{p_t^{*}}{1-\delta}\Big)',\text{ if }1\leq p_t<N,\\
         l\text{ for some }l>1,\text{ if }p_t=N,\\
         1\text{ if }p_t=N.
     \end{cases}
\end{equation}

\item The space $W_0^{1,p}(\Omega,w)$ will be denoted by $X_\alpha$, where $\alpha>0$ and for $u\in X_\alpha$, we recall the norm $\|u\|_{X_\alpha}$, as defined in \eqref{norm}.

    \item For \(l>1\), we define
    $
    J_l(m):=|m|^{l-2}m,
    \qquad m\in\mathbb{R}.
    $
    \item We recall the notation  
    $
    d\mu:=\frac{K(x,y)}{|x-y|^{sp}}\,dx\,dy
    $
    from \eqref{dmu}.
    \item For an open set \(\omega\), the notation
    $
    \omega\Subset\Omega
    $
    means that \(\omega\) is compactly contained in \(\Omega\), i.e.,
    $
    \overline{\omega}\subset\Omega.
    $
    \item For \(k\in\mathbb{R}\), we use the standard notation
    \[
    k^{+}:=\max\{k,0\},\qquad
    k^{-}:=\max\{-k,0\},\qquad
    k_{-}:=\min\{k,0\}.
    \]

    \item If \(F\) is a measurable function on a measurable set \(S\), then the notation
    $
    c\le F\le d
    \quad\text{in }S
    $
    means that
    $
    c\le F(x)\le d
   \text{ for almost every }x\in S.
    $ We also use the abbreviation a.e. to mean almost every.
    
    \item The letter \(C\) denotes a generic positive constant, whose value may vary from one occurrence to another. Whenever its dependence on certain parameters is relevant, we write
    $
    C=C(r_1,r_2,\ldots,r_k).
    $
\end{itemize}

\subsection{Known results}
The weighted Sobolev embedding theorem stated below can be found in, for instance \cite[Theorem~2.6]{Garainmn}, which will play a crucial role in our analysis. We also refer to \cite{Drabek}.

To state the weighted Sobolev embedding theorem, we introduce the following subclass of Muckenhoupt weights. For $1<p<\infty$ and $N\geq 2$, we define the following set
\[
I
:=
\left(\frac{1}{p-1},\infty\right)
\cap
\left(\frac{N}{p},\infty\right),
\]
and we define the following subclass of $A_p$ by
\begin{equation}\label{w}
A_p^{t}
:=
\left\{
w\in A_p:
w^{-t}\in L^1(\Omega)
\text{ for some }
t\in I
\right\}.
\end{equation}

For example, for given $t\in I$, the power weight
$
w(x)=|x|^\alpha\in A_p^{t},
$
if
$
-\frac{N}{t}<\alpha<\frac{N}{t}.
$

\begin{Theorem}[Weighted Sobolev embedding]\label{emb}
Let $1<p<\infty$ and let $w\in A_p^t$ for some $t\in I$. Then the following continuous embeddings hold:
\[
W^{1,p}(\Omega,w)
\hookrightarrow
W^{1,p_t}(\Omega)
\hookrightarrow
\begin{cases}
L^{q}(\Omega), & p_t\le q\le p_t^{*}, \quad \text{if } 1\le p_t<N,\\[1mm]
L^{q}(\Omega), & 1\le q<\infty, \quad \text{if } p_t=N,\\[1mm]
C(\overline{\Omega}), & \text{if } p_t>N,
\end{cases}
\]
where
$
p_t:=\frac{pt}{t+1}\in[1,p),
$
and
$
p_t^{*}:=\frac{Np_t}{N-p_t}
$
whenever \(p_t<N\).

Moreover, the second embedding above is compact except in the critical case
$
q=p_t^{*},
$
when \(1\le p_t<N\). The same embedding results remain valid for the space
\(W_0^{1,p}(\Omega,w)\).
\end{Theorem}

\begin{Remark}\label{rmk:embedding}
If there exist positive constants \(c\) and \(d\) such that
$
0<c\le w(x)\le d
\text{ for a.e. }x\in\Omega,
$
then
$
W^{1,p}(\Omega,w)=W^{1,p}(\Omega)
$
with equivalent norms. Consequently, Theorem~\ref{emb} reduces to the classical Sobolev embedding theorem, with \(p_t\) replaced by \(p\).
\end{Remark}

The following result is taken from \cite[Theorem 9.14]{MB}.

\begin{Theorem}\label{MBthm}
Let $V$ be a real separable reflexive Banach space and $V^*$ be the dual of $V$. Suppose that
$
T : V \to V^*
$
is a coercive and demicontinuous monotone operator. Then $T$ is surjective; that is, given any
$f \in V^*$, there exists $u \in V$ such that
$
T(u) = f.
$
If $T$ is strictly monotone, then $T$ is also injective.
\end{Theorem}

For the following result, see \cite[Lemma B.1]{Stam}.
\begin{Lemma}\label{Stamlem}
Let $\phi:[k_0,\infty)\to[0,\infty)$ be a nonnegative and nonincreasing function. Suppose that there exist positive constants $c$, $\ell$, and $m$, with $m>1$, such that
\[
\phi(h)\leq \frac{c}{(h-k)^{\ell}}\,\bigl[\phi(k)\bigr]^m,
\qquad \text{for all } h>k> k_0.
\]
Then
\[
\phi(k_0+d)=0,
\]
where
\[
d^{\ell}
=
c\,\bigl[\phi(k_0)\bigr]^{m-1}
2^{\frac{\ell m}{m-1}}.
\]
\end{Lemma}

Next, we recall the following algebraic inequality from \cite[Lemma 2.1]{Dama}.

\begin{Lemma}\label{alg}
Let \(1<p<\infty\). Then there exists a constant \(C=C(p)>0\) such that, for every \(a,b\in\mathbb{R}^N\),
\begin{equation}\label{eq:algebraic-inequality}
\left\langle |a|^{p-2}a-|b|^{p-2}b,\; a-b\right\rangle
\ge
C\,\frac{|a-b|^{2}}{(|a|+|b|)^{\,2-p}}.
\end{equation}
\end{Lemma}

\section{Preliminaries}
For $n\in\mathbb{N}$ and a nonnegative function $f\in L^{1}(\Omega)\setminus\{0\}$, we define
\[
f_n(x):=\min\{f(x),n\}.
\]
We consider the following approximate problem:
\begin{equation}\label{approxeqn}
\mathcal{M}_\alpha u
=
f_n(x)\left(u^{+}+\frac{1}{n}\right)^{-\delta}
\quad \text{in }\Omega,
\qquad
u=0
\quad \text{in }\mathbb{R}^{N}\setminus\Omega.
\end{equation}
The main objective of this section is to establish the existence of a weak solution $u_n$ to \eqref{approxeqn} and to derive suitable \emph{a priori} estimates that will be used in the subsequent analysis. While the following result in Lemma \ref{approx} is stated, in accordance with our standing assumptions, only for $0<\delta<1$, we point out that it actually holds for any $\delta>0$.

\begin{Lemma}\label{approx}
The following assertions hold:
\begin{enumerate}
\item[(i)] For every $n\in\mathbb{N}$, the problem \eqref{approxeqn} admits a unique positive weak solution
$
u_n\in X_\alpha\cap L^{\infty}(\Omega).
$

\item[(ii)] The sequence $\{u_n\}_{n\in\mathbb{N}}$ is monotone increasing, that is,
$
u_n\leq u_{n+1}\quad\text{in }\Omega,\qquad \forall\,n\in\mathbb{N}.
$

\item[(iii)] Moreover, for every $\omega\Subset\Omega$, there exists a constant $C(\omega)>0$, independent of $n$, such that
$
u_n(x)\geq C(\omega)\text{ for a.e. }x\in\omega,\ \forall\,n\in\mathbb{N}.
$
\end{enumerate}
\end{Lemma}
\begin{proof}
We recall the notation $X_\alpha=W_{0}^{1,p}(\Omega,w)$. We denote the dual space of $X_\alpha$ by $X_\alpha^{*}$. We define the operator $S:X_\alpha\to X_\alpha^{*}$ by
\begin{align*}
\langle S(v),\varphi\rangle
&:= \int_{\Omega}H(\nabla v)^{p-1}\nabla H(\nabla v)\nabla\varphi\,w\,dx \\
&\quad +\alpha \int_{\mathbb{R}^{N}}\int_{\mathbb{R}^{N}}
J_{p}(v(x)-v(y))\bigl(\varphi(x)-\varphi(y)\bigr)\,d\mu \\
&\quad - \int_{\Omega}f_n(x)\left(v^{+}+\frac{1}{n}\right)^{-\delta}\varphi\,dx,
\end{align*}
for all $v,\varphi\in X_\alpha$. By Lemma~\ref{emb}, together with H\"older's inequality, it is readily verified that $S$ is a well-defined operator from $X_\alpha$ into $X_\alpha^{*}$.

\begin{itemize}
\item \textbf{Coercivity:}
By Lemma~\ref{emb} and H\"older's inequality, we obtain
\[
\langle S(v),v\rangle
=
\|v\|_{X_\alpha}^{p}
-
\int_{\Omega}
f_n(x)\left(v^{+}+\frac{1}{n}\right)^{-\delta}v\,dx
\ge
\|v\|_{X_\alpha}^{p}-C\|v\|_{X_\alpha},
\]
for some constant \(C>0\) independent of \(v\). Consequently,
\[
\frac{\langle S(v),v\rangle}{\|v\|_{X_\alpha}}
\ge
\|v\|_{X_\alpha}^{\,p-1}-C
\longrightarrow\infty
\quad\text{as }\|v\|_{X_\alpha}\to\infty.
\]
Since \(1<p<\infty\), it follows that \(S\) is coercive.

\item \textbf{Demicontinuity:}
Let $\{v_k\}_{k\in\mb{N}}\subset X_\alpha$ be such that
$
\|v_k-v\|_{X_\alpha}\longrightarrow0
\quad\text{as }k\to\infty.
$
By Lemma~\ref{emb}, passing to a subsequence if necessary, we may assume that
$
v_k(x)\to v(x)\quad\text{a.e. in }\Omega,
$
and consequently
$
v_k(x)-v_k(y)\to v(x)-v(y)
\quad\text{for a.e. }(x,y)\in\mathbb{R}^{2N}.
$

Taking this pointwise convergence into account along with the boundedness of
\[
\left\{
\frac{J_p(v_k(x)-v_k(y))}{|x-y|^{\frac{sp}{p'}}}
\frac{w(x)^{1/p'}w(y)^{1/p'}}{w(B_{x,y})^{1/p'}}
\right\}_{k\in\mathbb N}
\]
in $L^{p'}(\mathbb{R}^{2N})$, we arrive at
\[
\frac{J_p(v_k(x)-v_k(y))}{|x-y|^{\frac{sp}{p'}}}
\frac{w(x)^{1/p'}w(y)^{1/p'}}{w(B_{x,y})^{1/p'}}
\rightharpoonup
\frac{J_p(v(x)-v(y))}{|x-y|^{\frac{sp}{p'}}}
\frac{w(x)^{1/p'}w(y)^{1/p'}}{w(B_{x,y})^{1/p'}}
\]
weakly in $L^{p'}(\mathbb{R}^{2N})$. Since
\[
\frac{\varphi(x)-\varphi(y)}{|x-y|^{\frac{sp}{p}}}
\frac{w(x)^{1/p}w(y)^{1/p}}{w(B_{x,y})^{1/p}}
\in L^{p}(\mathbb{R}^{2N}),
\]
it follows that
\begin{equation}\label{demi1}
\lim_{k\to\infty}
\int_{\mathbb{R}^{N}}\int_{\mathbb{R}^{N}}
J_p(v_k(x)-v_k(y))
(\varphi(x)-\varphi(y))
\,d\mu
=\int_{\mathbb{R}^{N}}\int_{\mathbb{R}^{N}}
J_p(v(x)-v(y))
(\varphi(x)-\varphi(y))
\,d\mu.
\end{equation}

Since
\[
\nabla v_k\to\nabla v
\quad\text{in }L^p(\Omega,w),
\]
we have, up to a subsequence,
\[
\nabla v_k(x)\to\nabla v(x)
\quad\text{a.e. in }\Omega.
\]
Therefore,
\[
w^{1/p'}
H(\nabla v_k)^{p-1}\nabla H(\nabla v_k)
\to
w^{1/p'}
H(\nabla v)^{p-1}\nabla H(\nabla v)
\quad\text{a.e. in }\Omega.
\]
Furthermore,
\[
\left\|
w^{1/p'}
H(\nabla v_k)^{p-1}\nabla H(\nabla v_k)
\right\|_{L^{p'}(\Omega)}^{p'}
\leq C,
\]
so, by weak compactness,
\[
w^{1/p'}
H(\nabla v_k)^{p-1}\nabla H(\nabla v_k)
\rightharpoonup
w^{1/p'}
H(\nabla v)^{p-1}\nabla H(\nabla v)
\quad\text{weakly in }L^{p'}(\Omega).
\]
Consequently,
\begin{equation}\label{eq:limit_identity}
\lim_{k\to\infty}
\int_{\Omega}
H(\nabla v_k)^{p-1}
\nabla H(\nabla v_k)\nabla\varphi\,w\,dx
=
\int_{\Omega}
H(\nabla v)^{p-1}
\nabla H(\nabla v)\nabla\varphi\,w\,dx.
\end{equation}

Finally, since
\[
0\le
f_n(x)\left(v_k^++\frac1n\right)^{-\delta}
\le
n^{1+\delta},
\]
the Lebesgue's dominated convergence theorem yields
\begin{equation}\label{dctap1}
\lim_{k\to\infty}
\int_{\Omega}
f_n(x)
\left(v_k^++\frac1n\right)^{-\delta}
\varphi\,dx
=
\int_{\Omega}
f_n(x)
\left(v^++\frac1n\right)^{-\delta}
\varphi\,dx,
\qquad
\forall\,\varphi\in X_\alpha.
\end{equation}

Combining \eqref{demi1}, \eqref{eq:limit_identity}, and \eqref{dctap1}, we conclude that
\[
\lim_{k\to\infty}\langle S(v_k),\varphi\rangle
=
\langle S(v),\varphi\rangle,
\qquad
\forall\,\varphi\in X_\alpha,
\]
and hence the operator $S:X_\alpha\to X_\alpha^{*}$ is demicontinuous.

\item \textbf{Monotonicity:}
Let $u_1,u_2\in X_\alpha$. Then
\begin{align*}
&\langle S(u_1)-S(u_2),u_1-u_2\rangle \\
&=
\int_{\Omega}
\Big(
H(\nabla u_1)^{p-1}\nabla H(\nabla u_1)
-
H(\nabla u_2)^{p-1}\nabla H(\nabla u_2)
\Big)(\nabla u_1-\nabla u_2)\,w\,dx \\
&\quad
+\int_{\mathbb{R}^{N}}\int_{\mathbb{R}^{N}}
\Big(
J_p(u_1(x)-u_1(y))
-
J_p(u_2(x)-u_2(y))
\Big)
\Big(
(u_1-u_2)(x)-(u_1-u_2)(y)
\Big)\,d\mu \\
&\quad
-\int_{\Omega}
f_n(x)
\left[
\left(u_1^{+}+\frac1n\right)^{-\delta}
-
\left(u_2^{+}+\frac1n\right)^{-\delta}
\right]
(u_1-u_2)\,dx.
\end{align*}
By the property \eqref{alg2}, the first integral above is nonnegative. Moreover, by Lemma~\ref{alg}, the second integral is also nonnegative. We observe that the third integral above is nonpositive. Consequently,
\[
\langle S(u_1)-S(u_2),u_1-u_2\rangle\ge 0,
\]
which shows that the operator $S$ is monotone.
\end{itemize}

By Theorem~\ref{MBthm}, the operator $S:X_\alpha\to X_\alpha^{*}$ is surjective. Consequently, for every $n\in\mathbb{N}$, there exists a function $u_n\in X_\alpha$ satisfying
\begin{equation}\label{auxeqn}
\begin{split}
&\int_{\Omega}
H(\nabla u_n)^{p-1}\nabla H(\nabla u_n)\nabla\varphi\,w\,dx 
+\int_{\mathbb{R}^{N}}\int_{\mathbb{R}^{N}}
J_p\bigl(u_n(x)-u_n(y)\bigr)
\bigl(\varphi(x)-\varphi(y)\bigr)\,d\mu \\
&\quad
=
\int_{\Omega}
f_n(x)
\left(u_n^{+}+\frac{1}{n}\right)^{-\delta}
\varphi\,dx,
\end{split}
\end{equation}
for every $\varphi\in X_\alpha$.

\medskip
\noindent \textbf{Nonnegativity.}
Taking $\varphi=(u_n)_{-}:=\min\{u_n,0\}$ as a test function in \eqref{auxeqn}, we obtain
\begin{equation}\label{posi}
\begin{split}
&\int_{\Omega}
H(\nabla u_n)^{p-1}\nabla H(\nabla u_n)\nabla (u_n)_{-}\,w\,dx
+\int_{\mathbb{R}^{N}}\int_{\mathbb{R}^{N}}
J_p(u_n(x)-u_n(y))
\bigl((u_n)_{-}(x)-(u_n)_{-}(y)\bigr)\,d\mu \\
&\quad
=
\int_{\Omega}
f_n(x)
\left(u_n^{+}+\frac1n\right)^{-\delta}
(u_n)_{-}\,dx
\le0.
\end{split}
\end{equation}

Moreover, by \cite[Equation~(3.13), page~12]{GU21},
\begin{equation}\label{pos}
J_p(u_n(x)-u_n(y))
\bigl((u_n)_{-}(x)-(u_n)_{-}(y)\bigr)
\ge
\bigl|(u_n)_{-}(x)-(u_n)_{-}(y)\bigr|^{p},
\end{equation}
for every $x,y\in\mathbb{R}^{N}$.

Combining \eqref{posi} and \eqref{pos}, and using Lemma \ref{prop}, we infer that
$
\|(u_n)_{-}\|_{X_\alpha}^{p}=0.
$
Hence,
$
(u_n)_{-}\equiv0,
$
which shows that
\begin{equation}\label{non-neg}
u_n\ge0
\quad\text{a.e. in }\mathbb{R}^{N},
\qquad
\forall\,n\in\mathbb{N}.
\end{equation}

\textbf{Monotonicity and uniqueness.}
Let $n\in\mathbb{N}$ be fixed, and let $u_n,u_{n+1}\in X_\alpha$ denote the corresponding weak solutions of \eqref{approxeqn}. By \eqref{non-neg}, both $u_n$ and $u_{n+1}$ are nonnegative in $\mathbb{R}^{N}$. Moreover, for every $\varphi\in X_\alpha$, we have
\begin{equation}\label{auxeqn11}
\begin{aligned}
&\int_{\Omega}
H(\nabla u_n)^{p-1}\nabla H(\nabla u_n)\nabla\varphi\,w\,dx \\
&\quad
+\int_{\mathbb{R}^{N}}\int_{\mathbb{R}^{N}}
J_p(u_n(x)-u_n(y))
(\varphi(x)-\varphi(y))\,d\mu
=
\int_{\Omega}
f_n(x)
\left(u_n+\frac1n\right)^{-\delta}
\varphi\,dx,
\end{aligned}
\end{equation}
and
\begin{equation}\label{auxeqn21}
\begin{aligned}
&\int_{\Omega}
H(\nabla u_{n+1})^{p-1}\nabla H(\nabla u_{n+1})\nabla\varphi\,w\,dx \\
&\quad
+\int_{\mathbb{R}^{N}}\int_{\mathbb{R}^{N}}
J_p(u_{n+1}(x)-u_{n+1}(y))
(\varphi(x)-\varphi(y))\,d\mu
=
\int_{\Omega}
f_{n+1}(x)
\left(u_{n+1}+\frac1{n+1}\right)^{-\delta}
\varphi\,dx.
\end{aligned}
\end{equation}

Taking
\[
\varphi=(u_n-u_{n+1})^{+}\in X_\alpha
\]
as a test function in both \eqref{auxeqn11} and \eqref{auxeqn21}, subtracting the resulting identities, and using the fact that $f_n\le f_{n+1}$ a.e. in $\Omega$, we obtain
\[
\int_{\Omega}
f_n\left(u_n+\frac1n\right)^{-\delta}\varphi\,dx
-
\int_{\Omega}
f_{n+1}\left(u_{n+1}+\frac1{n+1}\right)^{-\delta}\varphi\,dx
\le0.
\]
Consequently,
\begin{equation}\label{in-neg}
\begin{aligned}
&\int_{\Omega}
\Big(
H(\nabla u_n)^{p-1}\nabla H(\nabla u_n)
-
H(\nabla u_{n+1})^{p-1}\nabla H(\nabla u_{n+1})
\Big)
\nabla\varphi\,w\,dx \\
&\quad
+\int_{\mathbb{R}^{N}}\int_{\mathbb{R}^{N}}
\Big(
J_p(u_n(x)-u_n(y))
-
J_p(u_{n+1}(x)-u_{n+1}(y))
\Big)
(\varphi(x)-\varphi(y))
\,d\mu
\le0.
\end{aligned}
\end{equation}

Arguing as in the proof of \cite[Lemma~9]{LL}, we have
\[
\Big(
J_p(u_n(x)-u_n(y))
-
J_p(u_{n+1}(x)-u_{n+1}(y))
\Big)
(\varphi(x)-\varphi(y))
\ge0
\]
for a.e. $(x,y)\in\mathbb{R}^{2N}$.
Using the property \eqref{alg2}, we arrive at
\[
\Big(H(\nabla u_n)^{p-1}\nabla H(\nabla u_n)
-
H(\nabla u_{n+1})^{p-1}\nabla H(\nabla u_{n+1})
\Big)\nabla (u_n-u_{n+1})^{+}\,w\,dx\geq 0.
\]
Combining the previous two inequalities into \eqref{in-neg}, we arrive at
$
(u_n-u_{n+1})^{+}=0
\quad\text{in }\Omega,
$
that is,
$
u_n\le u_{n+1}
\quad\text{a.e. in }\Omega.
$

The uniqueness of the weak solution $u_n$ follows by an analogous argument.

\noindent
\textbf{Boundedness.}
To establish the boundedness of \(u_n\), for each \(k\ge1\), we define the set
$
A(k):=\{x\in\Omega:\,u_n(x)\ge k\},
$
and choose
$
\varphi_k:=(u_n-k)^+\in X_\alpha
$
as a test function in \eqref{auxeqn}. By using Lemma \ref{prop} and \cite[Equation~(3.9), page~11]{GU21}, we obtain
\begin{equation}\label{tstbd1}
\|\varphi_k\|_{X_\alpha}^{p}
\le
C\int_{\Omega}
f_n(x)
\left(u_n+\frac1n\right)^{-\delta}
\varphi_k\,dx,
\end{equation}
where \(C>0\) is independent of \(k\). Since \(f_n\le n\) and
$
\left(u_n+\frac1n\right)^{-\delta}
\le
n^\delta,
$
it follows from Lemma~\ref{emb} and the continuous embedding
$
X_\alpha\hookrightarrow L^{l}(\Omega),\,l>p,
$
that
\begin{align}
\|\varphi_k\|_{X_\alpha}^{p}
&\le
Cn^{1+\delta}
\int_{A(k)}(u_n-k)\,dx \notag\\
&\le
Cn^{1+\delta}
|A(k)|^{\frac{l-1}{l}}
\|\varphi_k\|_{L^{l}(\Omega)} \notag\le
C|A(k)|^{\frac{l-1}{l}}
\|\varphi_k\|_{X_\alpha},
\label{tstbd2}
\end{align}
where \(C>0\) depends on \(n\), but is independent of \(k\). Consequently,
\begin{equation}\label{new}
\|\varphi_k\|_{X_\alpha}^{p}
\le
C|A(k)|^{\frac{p(l-1)}{l(p-1)}}.
\end{equation}

Now let \(h>k\ge1\). Since
$
u_n-k\ge h-k
\quad\text{on }A(h),
$
and \(A(h)\subset A(k)\), using \eqref{new}, we obtain
\begin{align*}
(h-k)^p|A(h)|^{\frac{p}{l}}
&\le
\left(
\int_{A(h)}(u_n-k)^l\,dx
\right)^{\frac{p}{l}}\\
&\le
\left(
\int_{A(k)}(u_n-k)^l\,dx
\right)^{\frac{p}{l}}
\le
C\|\varphi_k\|_{X_\alpha}^{p}\le
C|A(k)|^{\frac{p(l-1)}{l(p-1)}}.
\end{align*}
Hence,
\[
|A(h)|
\le
\frac{C}{(h-k)^l}
|A(k)|^{\frac{l-1}{p-1}}.
\]
Since
$
\frac{l-1}{p-1}>1,
$
an application of Lemma \ref{Stamlem} yields
$
\|u_n\|_{L^\infty(\Omega)}
\le C,
$
where \(C>0\) depends on \(n\). Therefore,
$
u_n\in L^\infty(\Omega).
$

\medskip

\noindent\textbf{Uniform positivity.}
By \eqref{non-neg}, we have \(u_1\ge0\) in \(\mathbb{R}^{N}\). Since \(u_1=0\) in \(\mathbb{R}^{N}\setminus\Omega\) and \(f\not\equiv0\), it follows that \(u_1\not\equiv0\) in \(\Omega\). Therefore, by \cite[Theorem~1.6]{BG26}, for every compact set \(\omega\Subset\Omega\), there exists a constant \(C(\omega)>0\) such that
$
u_1\ge C(\omega)
\text{ in }\omega.
$
Since the sequence \(\{u_n\}_{n\in\mb{N}}\) is monotone increasing,
$
u_n\ge u_1\text{ in }\Omega,
\,\forall\,n\in\mathbb N.
$
Consequently,
$
u_n(x)\ge C(\omega),
\text{ for a.e. }
x\in\omega,\,n\in\mathbb N,
$
where the constant \(C(\omega)>0\) is independent of \(n\).
\end{proof}

\begin{Remark}\label{rmkapprox}
It follows from Lemma~\ref{approx} that the monotone increasing sequence $\{u_n\}_{n\in\mathbb{N}}$ converges pointwise in $\Omega$. We denote its limit by
\[
u_\delta:=\lim_{n\to\infty}u_n.
\]
In particular,
$
u_\delta\ge u_n
\text{ in }\Omega,\, \forall\,n\in\mathbb{N}.
$
We show that $u_\delta$ will be the required weak solution of the problem \eqref{meqn}.
\end{Remark}

\begin{Remark}\label{altrmk}  
We point out that Lemma~\ref{approx} can alternatively be proved by means of Schauder's fixed point theorem; see, for instance, \cite[Lemma~3.1]{GU21}. For this purpose, the following auxiliary result in Lemma \ref{auxresult}, which concerns a more general nonsingular problem, will be useful. Its proof follows the same lines as that of Lemma~\ref{approx}, the only essential modification being the definition of the operator
\[
S:X_\alpha\to X_\alpha^{*},
\]
given by
\[
\langle S(v),\varphi\rangle
:=
\int_{\Omega}
H(\nabla v)^{p-1}\nabla H(\nabla v)\nabla\varphi\,w\,dx
+\int_{\mathbb{R}^{N}}\int_{\mathbb{R}^{N}}
J_p(v(x)-v(y))
\bigl(\varphi(x)-\varphi(y)\bigr)\,d\mu
-
\int_{\Omega}g\,\varphi\,dx,
\]
for all \(v,\varphi\in X_\alpha\).
\end{Remark}

\begin{Lemma}\label{auxresult}
Let \(g\in L^{\infty}(\Omega)\setminus\{0\}\) be a nonnegative function. Then the problem
\begin{equation}\label{approxnew}
\mathcal{M}_{\alpha}u=g
\quad\text{in }\Omega,\qquad
u>0
\quad\text{in }\Omega,\qquad
u=0
\quad\text{in }\mathbb{R}^{N}\setminus\Omega,
\end{equation}
admits a unique weak solution
$
u\in X_\alpha\cap L^{\infty}(\Omega).
$
Furthermore, for every compact set \(\omega\Subset\Omega\), there exists a constant \(C(\omega)>0\), depending only on \(\omega\), such that
$
u\ge C(\omega)
\quad\text{in }\omega.
$
\end{Lemma}

\begin{Lemma}\label{apun2}
Then there exists a constant \(C>0\), independent of \(n\), such that
$
\|u_n\|_{X_\alpha}\le C,
\,\forall\,n\in\mathbb{N}.
$
In particular, the sequence \(\{u_n\}_{n\in\mathbb{N}}\) is uniformly bounded in \(X_\alpha\).
\end{Lemma}

\begin{proof}
First we note that \[
\left(u_n+\frac1n\right)^{-\delta}u_n
\le
u_n^{\,1-\delta},
\]
and \(f_n\le f\). Therefore, taking \(u_n\) as a test function in \eqref{auxeqn} and further using Lemma \ref{prop}, we obtain
\[
\|u_n\|_{X_\alpha}^{p}
\le
\int_{\Omega}
f\,u_n^{\,1-\delta}\,dx.
\]

Assume first that \(p_t<N\). Since \(f\in L^{m}(\Omega)\) and
$
(1-\delta)m'=p_t^{*},
$
H\"older's inequality together with Lemma~\ref{emb} yields
\begin{align*}
\|u_n\|_{X_\alpha}^{p}
&\le
\|f\|_{L^{m}(\Omega)}
\left(
\int_{\Omega}
u_n^{(1-\delta)m'}\,dx
\right)^{\frac1{m'}} \\
&=
\|f\|_{L^{m}(\Omega)}
\left(
\int_{\Omega}
u_n^{p_t^{*}}\,dx
\right)^{\frac{1-\delta}{p_t^{*}}} \\
&\le
C
\|f\|_{L^{m}(\Omega)}
\|u_n\|_{X_\alpha}^{\,1-\delta},
\end{align*}
where \(C>0\) is independent of \(n\). Hence,
$
\|u_n\|_{X_\alpha}\le C,
$
for some constant \(C>0\) independent of \(n\).

If \(p_t\ge N\), the conclusion follows analogously by using the corresponding embedding result in Lemma~\ref{emb}. Therefore, the sequence \(\{u_n\}_{n\in\mathbb{N}}\) is uniformly bounded in \(X_\alpha\).
\end{proof}

\begin{Lemma}\label{lemma1}
For every \(n\in\mathbb{N}\) and every \(\varphi\in X_\alpha\), the following estimate holds:
\begin{equation}\label{prop1}
\|u_n\|_{X_\alpha}^{p}
\le
\|\varphi\|_{X_\alpha}^{p}
+
p\int_{\Omega}
(u_n-\varphi)
\left(u_n+\frac1n\right)^{-\delta}
f_n\,dx.
\end{equation}
Furthermore, the sequence \(\{\|u_n\|_{X_\alpha}\}_{n\in\mathbb{N}}\) is nondecreasing, that is,
\begin{equation}\label{nmon}
\|u_n\|_{X_\alpha}
\le
\|u_{n+1}\|_{X_\alpha},
\qquad
\forall\,n\in\mathbb{N}.
\end{equation}
\end{Lemma}

\begin{proof}
Let \(h\in X_\alpha\). By Lemma~\ref{auxresult}, there exists a unique weak solution
$
v\in X_\alpha
$
of the problem
\[
\mathcal{M}_{\alpha}v
=
f_n(x)\left(h^{+}+\frac1n\right)^{-\delta}
\quad\text{in }\Omega,
\qquad
v>0
\quad\text{in }\Omega,
\qquad
v=0
\quad\text{in }\mathbb{R}^{N}\setminus\Omega.
\]
Moreover, we observe that \(v\) is the unique minimizer of the functional
\[
J:X_\alpha\to\mathbb{R},
\]
defined by
\[
J(\varphi)
:=
\frac1p\|\varphi\|_{X_\alpha}^{p}
-
\int_{\Omega}
f_n
\left(h^{+}+\frac1n\right)^{-\delta}
\varphi\,dx.
\]
Therefore,
\[
J(v)\le J(\varphi),
\qquad
\forall\,\varphi\in X_\alpha,
\]
which is equivalent to
\begin{equation}\label{mineqn}
\frac1p\|v\|_{X_\alpha}^{p}
-
\int_{\Omega}
f_n
\left(h^{+}+\frac1n\right)^{-\delta}
v\,dx
\le
\frac1p\|\varphi\|_{X_\alpha}^{p}
-
\int_{\Omega}
f_n
\left(h^{+}+\frac1n\right)^{-\delta}
\varphi\,dx.
\end{equation}

Choosing \(v=h=u_n\) in \eqref{mineqn} immediately yields \eqref{prop1}. Finally, taking \(\phi=u_{n+1}\) in \eqref{prop1} and using the monotonicity property \(u_n\le u_{n+1}\) established in Lemma~\ref{approx}, we obtain
$
\|u_n\|_{X_\alpha}\le \|u_{n+1}\|_{X_\alpha},
$
which proves \eqref{nmon}.
\end{proof}

\begin{Lemma}\label{strong}
Up to a subsequence, the sequence \(\{u_n\}_{n\in\mathbb{N}}\) converges strongly to \(u_\delta\) in \(X_\alpha\).
\end{Lemma}

\begin{proof}
Since \(u_n\le u_\delta\) in \(\Omega\), taking \(\phi=u_\delta\) in \eqref{prop1} yields
\[
\|u_n\|_{X_\alpha}\le \|u_\delta\|_{X_\alpha},
\qquad \forall\,n\in\mathbb{N}.
\]
Together with the monotonicity property \eqref{nmon} established in Lemma~\ref{lemma1}, this implies
\begin{equation}\label{lim1}
\lim_{n\to\infty}\|u_n\|_{X_\alpha}
\le
\|u_\delta\|_{X_\alpha}.
\end{equation}

On the other hand, by Lemma~\ref{apun2}, the sequence \(\{u_n\}_{n\in\mathbb{N}}\) is uniformly bounded in \(X_\alpha\). Hence, passing to a subsequence if necessary,
$
u_n\rightharpoonup u_\delta
\text{ weakly in }X_\alpha.
$
By the weak lower semicontinuity of the norm,
\begin{equation}\label{lim2}
\|u_\delta\|_{X_\alpha}
\le
\liminf_{n\to\infty}\|u_n\|_{X_\alpha}.
\end{equation}
Combining \eqref{lim1} and \eqref{lim2}, we conclude that
$
\lim_{n\to\infty}\|u_n\|_{X_\alpha}
=
\|u_\delta\|_{X_\alpha}.
$
Since \(X_\alpha\), equipped with the norm \(\|\cdot\|_{X_\alpha}\), is uniformly convex, weak convergence together with convergence of the norms implies strong convergence. Therefore,
$
u_n\to u_\delta
\text{ strongly in }X_\alpha,
$
which completes the proof.
\end{proof}

\begin{Lemma}\label{minprop}
Define the functional \(I_\delta:X_\alpha\to\mathbb{R}\) by
\[
I_\delta(v)
:=
\frac1p\|v\|_{X_\alpha}^{p}
-
\frac{1}{1-\delta}
\int_{\Omega}(v^{+})^{1-\delta}f\,dx.
\]
Then \(u_\delta\) is a global minimizer of \(I_\delta\) over \(X_\alpha\).
\end{Lemma}

\begin{proof}
For each \(n\in\mathbb{N}\), consider the functional
$
I_n:X_\alpha\to\mathbb{R},
$
defined by
\[
I_n(v)
:=
\frac1p\|v\|_{X_\alpha}^{p}
-
\int_{\Omega}G_n(v)f_n\,dx,
\]
where
\[
G_n(t)
:=
\frac{1}{1-\delta}
\left(t^{+}+\frac1n\right)^{1-\delta}
-
\left(\frac1n\right)^{-\delta}t^{-}.
\]
It is readily verified that \(I_n\) is coercive, weakly lower semicontinuous, and of class \(C^{1}\). Consequently, \(I_n\) admits a global minimizer \(v_n\in X_\alpha\), which satisfies
$
\langle I_n'(v_n),\phi\rangle=0,
\,
\forall\,\phi\in X_\alpha.
$

Since
$
I_n(v_n)\le I_n(v_n^{+}),
$
we necessarily have \(v_n\ge0\) in \(\Omega\). Therefore, \(v_n\) is a weak solution of the approximate problem \eqref{approxeqn}. By the uniqueness statement in Lemma~\ref{approx}, it follows that
$
v_n=u_n.
$
Hence, \(u_n\) is the unique minimizer of \(I_n\), and therefore
\begin{equation}\label{min}
I_n(u_n)\le I_n(v^{+}),
\qquad
\forall\,v\in X_\alpha.
\end{equation}

Since \(u_n\uparrow u_\delta\) a.e. in \(\Omega\) and \(f_n\uparrow f\), the Lebesgue's dominated convergence theorem yields
\[
\lim_{n\to\infty}
\int_{\Omega}G_n(u_n)f_n\,dx
=
\frac{1}{1-\delta}
\int_{\Omega}
u_\delta^{\,1-\delta}f\,dx.
\]
Moreover, by Lemma~\ref{strong},
$
u_n\to u_\delta
\quad\text{strongly in }X_\alpha,
$
and hence
$
\lim_{n\to\infty}\|u_n\|_{X_\alpha}
=
\|u_\delta\|_{X_\alpha}.
$
Consequently,
\begin{equation}\label{newlim2}
\lim_{n\to\infty}I_n(u_n)
=
I_\delta(u_\delta).
\end{equation}

On the other hand, for every \(v\in X_\alpha\), another application of the Lebesgue's dominated convergence theorem gives
\begin{equation}\label{lim3}
\lim_{n\to\infty}
\int_{\Omega}G_n(v^{+})f_n\,dx
=
\frac{1}{1-\delta}
\int_{\Omega}(v^{+})^{1-\delta}f\,dx.
\end{equation}

Since \(\|v^{+}\|_{X_\alpha}\le\|v\|_{X_\alpha}\), passing to the limit in \eqref{min} with the aid of \eqref{newlim2} and \eqref{lim3} yields
$
I_\delta(u_\delta)
\le
I_\delta(v),
\,
\forall\,v\in X_\alpha.
$
Thus, \(u_\delta\) is a minimizer of \(I_\delta\), completing the proof.
\end{proof}

\section{Proof of the main results}

\begin{proof}[Proof of Theorem \ref{thm1}]
By Lemma~\ref{strong}, passing to a further subsequence if necessary, we may assume that
\[
\nabla u_n(x)\to\nabla u_\delta(x)
\quad\text{for a.e. }x\in\Omega.
\]
Consequently,
\[
w(x)^{\frac1{p'}}
H(\nabla u_n(x))^{p-1}
\nabla H(\nabla u_n(x))
\to
w(x)^{\frac1{p'}}
H(\nabla u_\delta(x))^{p-1}
\nabla H(\nabla u_\delta(x))
\quad\text{for a.e. }x\in\Omega.
\]

Moreover,
\[
\left\|
w^{\frac1{p'}}
H(\nabla u_n)^{p-1}
\nabla H(\nabla u_n)
\right\|_{L^{p'}(\Omega)}^{p'}
\leq
C,
\]
where \(C>0\) is a constant independent of \(n\). Hence, the sequence
\[
\left\{
w^{\frac1{p'}}
H(\nabla u_n)^{p-1}
\nabla H(\nabla u_n)
\right\}_{n\in\mathbb{N}}
\]
is bounded in \(L^{p'}(\Omega)\). Therefore, by weak compactness and the pointwise convergence above,
\[
w^{\frac1{p'}}
H(\nabla u_n)^{p-1}
\nabla H(\nabla u_n)
\rightharpoonup
w^{\frac1{p'}}
H(\nabla u_\delta)^{p-1}
\nabla H(\nabla u_\delta)
\quad\text{weakly in }L^{p'}(\Omega).
\]

Let \(\varphi\in C_c^{1}(\Omega)\), then
$
w^{\frac1p}\nabla\varphi\in L^p(\Omega).
$
It therefore follows from the above weak convergence that
\begin{equation}\label{l1}
\lim_{n\to\infty}
\int_{\Omega}
H(\nabla u_n)^{p-1}
\nabla H(\nabla u_n)
\nabla\varphi\,w\,dx
=
\int_{\Omega}
H(\nabla u_\delta)^{p-1}
\nabla H(\nabla u_\delta)
\nabla\varphi\,w\,dx.
\end{equation}
Since $\phi \in C_c^{1}(\Omega)$ and the sequence $\{u_n\}_{n\in\mathbb{N}}$ is uniformly bounded in $X_\alpha$, it follows that
\[
\left\{
\frac{J_p(u_n(x)-u_n(y))}{|x-y|^{\frac{sp}{p'}}}
\frac{w(x)^{\frac1{p'}}w(y)^{\frac1{p'}}}{w(B_{x,y})^{\frac1{p'}}}
\right\}_{n\in\mathbb{N}}
\]
is bounded in $L^{p'}(\mathbb{R}^{2N})$. Moreover,
\[
\frac{\phi(x)-\phi(y)}{|x-y|^{\frac{sp}{p}}}
\frac{w(x)^{\frac1p}w(y)^{\frac1p}}{w(B_{x,y})^{\frac1p}}
\in L^{p}(\mathbb{R}^{2N}).
\]
Therefore, by the weak convergence established above, we obtain
\begin{equation}\label{l2}
\lim_{n\to\infty}
\int_{\mathbb{R}^{N}}\int_{\mathbb{R}^{N}}
J_p(u_n(x)-u_n(y))
\bigl(\phi(x)-\phi(y)\bigr)\,d\mu
=
\int_{\mathbb{R}^{N}}\int_{\mathbb{R}^{N}}
J_p(u_\delta(x)-u_\delta(y))
\bigl(\phi(x)-\phi(y)\bigr)\,d\mu.
\end{equation}
Moreover, by Lemma \ref{approx}, we have
$
u_n \geq c(\omega) > 0 \text{ for every } \omega \Subset \Omega.
$
Hence, for every $\phi \in C_c^1(\Omega)$,
\[
\left|\frac{f_n \phi}{\left(u_n+\frac{1}{n}\right)^\delta}\right|
\leq
\frac{\|\phi\|_{\infty}}{c(\omega)^\delta}\,f
\in L^1(\Omega).
\]
Furthermore,
\[
\frac{f_n}{\left(u_n+\frac{1}{n}\right)^\delta}\,\varphi
\longrightarrow
\frac{f}{u^\delta}\,\varphi
\quad \text{pointwise almost everywhere in } \Omega
\quad \text{as } n\to\infty.
\]
Therefore, by the Lebesgue's dominated convergence theorem,
\begin{equation}\label{r}
\lim_{n\to\infty}
\int_{\Omega}
\frac{f_n \varphi}{\left(u_n+\frac{1}{n}\right)^\delta}\,dx
=
\int_{\Omega}
\frac{f}{u^\delta}\,\varphi\,dx,\quad\forall\varphi\in C_c^{1}(\Omega).
\end{equation}
Finally, combining \eqref{l1}, \eqref{l2}, and \eqref{r} in the weak formulation of \eqref{auxeqn}, we conclude the desired result.

\textbf{Uniqueness:} Suppose that $v_1,v_2\in X_\alpha$ are two weak solutions of
problem~\eqref{meqn}. By Remark~\ref{Rmwksol}, both satisfy the weak formulation for every
$\varphi\in X_\alpha$, namely,
\begin{align}
&\int_{\Omega}
H(\nabla v_1)^{p-1}\nabla H(\nabla v_1)\nabla\varphi\,w\,dx
+\int_{\mathbb{R}^{N}}\int_{\mathbb{R}^{N}}
J_p(v_1(x)-v_1(y))
\bigl(\varphi(x)-\varphi(y)\bigr)\,d\mu
\nonumber\\
&\qquad=
\int_{\Omega}
f(x)v_1^{-\delta}\varphi\,dx,
\label{eq:weak-u1}
\end{align}
and
\begin{align}
&\int_{\Omega}
H(\nabla v_2)^{p-1}\nabla H(\nabla v_2)\nabla\varphi\,w\,dx
+\int_{\mathbb{R}^{N}}\int_{\mathbb{R}^{N}}
J_p(v_2(x)-v_2(y))
\bigl(\varphi(x)-\varphi(y)\bigr)\,d\mu
\nonumber\\
&\qquad=
\int_{\Omega}
f(x)v_2^{-\delta}\varphi\,dx.
\label{eq:weak-u2}
\end{align}

Choosing
$
\varphi=(v_1-v_2)^+\in X_\alpha
$
as a test function in \eqref{eq:weak-u1} and \eqref{eq:weak-u2}, and subtracting the two identities, we obtain
\[
\begin{aligned}
&\int_{\Omega}
\Big(H(\nabla v_1)^{p-1}\nabla H(\nabla v_1)
-
H(\nabla v_2)^{p-1}\nabla H(\nabla v_2)\Big)\nabla (v_1-v_2)^+\,w\,dx
\\
&\quad
+\int_{\mathbb{R}^{N}}\int_{\mathbb{R}^{N}}
\Big(J_p(v_1(x)-v_1(y))
-
J_p(v_2(x)-v_2(y))\Big)
\Big((v_1-v_2)^+(x)-(v_1-v_2)^+(y)\Big)\,d\mu
\\
&=
\int_{\Omega}
f(x)\bigl(v_1^{-\delta}-v_2^{-\delta}\bigr)
(v_1-v_2)^+\,dx.
\end{aligned}
\]
Now, arguing exactly as in the proof of monotonicity of $u_n$ in Lemma \ref{approx}, we obtain that
$
(v_1-v_2)^+(x)=0
\,\text{ for a.e. } x \text{ in }\Omega.
$
Interchanging the roles of $v_1$ and $v_2$ yields
$
(v_2-v_1)^+(x)=0
\,\text{ for a.e. } x \text{ in }\Omega.
$
Therefore,
$
v_1(x)=v_2(x)
\,\text{ for a.e. } x \text{ in }\Omega,
$
which completes the proof.
\end{proof}

\begin{proof}[Proof of Theorem \ref{thm5}]  
Recall that \(X_\alpha\) is endowed with the norm
\[
\|v\|_{X_\alpha}
=
\left(
\int_{\Omega}H(\nabla v)^{p}\,w\,dx
+
\alpha
\int_{\mathbb{R}^{N}}\int_{\mathbb{R}^{N}}
|v(x)-v(y)|^{p}\,d\mu
\right)^{\frac1p}.
\]

\begin{enumerate}
\item[$(a)$] We first prove that
\[
\Theta(\Omega)
:=
\inf_{v\in S_\delta}\|v\|_{X_\alpha}^{p}
=
\|u_\delta\|_{X_\alpha}^{\frac{p(1-\delta-p)}{1-\delta}},
\]
where
\[
S_\delta
:=
\left\{
v\in X_\alpha:
\int_{\Omega}|v|^{1-\delta}f\,dx=1
\right\}.
\]

Define
\[
V_\delta
:=
\tau_\delta u_\delta,
\qquad
\tau_\delta
:=
\left(
\int_{\Omega}
u_\delta^{\,1-\delta}f\,dx
\right)^{-\frac{1}{1-\delta}}.
\]
Then \(V_\delta\in S_\delta\).

By Remark~\ref{Rmwksol}, choosing \(u_\delta\) as a test function in \eqref{wksoleqn}, we obtain
$
\|u_\delta\|_{X_\alpha}^{p}
=
\int_{\Omega}
u_\delta^{\,1-\delta}f\,dx.
$
Consequently,
\[
\|V_\delta\|_{X_\alpha}^{p}
=
\tau_\delta^{p}\|u_\delta\|_{X_\alpha}^{p}
=
\left(
\int_{\Omega}
u_\delta^{\,1-\delta}f\,dx
\right)^{-\frac{p}{1-\delta}}
\|u_\delta\|_{X_\alpha}^{p}
=
\|u_\delta\|_{X_\alpha}^{\frac{p(1-\delta-p)}{1-\delta}}.
\]

Now let \(v\in S_\delta\) be arbitrary and we set
$
\lambda
=
\|v\|_{X_\alpha}^{-\frac{p}{p+\delta-1}}.
$
Since \(u_\delta\) minimizes the functional \(I_\delta\) (cf. Lemma~\ref{minprop}), we have
$
I_\delta(u_\delta)
\le
I_\delta(\lambda|v|).
$
Using the fact that \(v\in S_\delta\), a straightforward computation yields
\[
\|u_\delta\|_{X_\alpha}^{\frac{p(1-\delta-p)}{1-\delta}}
\le
\|v\|_{X_\alpha}^{p}.
\]
Taking the infimum over all \(v\in S_\delta\), we conclude that
\[
\Theta(\Omega)
=
\|u_\delta\|_{X_\alpha}^{\frac{p(1-\delta-p)}{1-\delta}},
\]
which completes the proof of part $(a)$.

\item[$(b)$] Suppose first that the weighted Sobolev-type inequality \eqref{inequality2} holds. If \(C>\Theta(\Omega)\), then, by part~$(a)$ and the fact that \(V_\delta\in S_\delta\),
\[
C\left(\int_{\Omega}V_\delta^{\,1-\delta}f\,dx\right)^{\frac{p}{1-\delta}}
>
\|V_\delta\|_{X_\alpha}^{p},
\]
which contradicts \eqref{inequality2}. Therefore,
$
C\le\Theta(\Omega).
$

Conversely, assume that \(C\le\Theta(\Omega)\). Let \(v\in X_\alpha\setminus\{0\}\) and we define
\[
w
:=
\left(
\int_{\Omega}|v|^{1-\delta}f\,dx
\right)^{-\frac{1}{1-\delta}}
v.
\]
Then \(w\in S_\delta\). Hence,
\[
C
\le
\Theta(\Omega)
\le
\|w\|_{X_\alpha}^{p}.
\]
Since
\[
\|w\|_{X_\alpha}^{p}
=
\left(
\int_{\Omega}|v|^{1-\delta}f\,dx
\right)^{-\frac{p}{1-\delta}}
\|v\|_{X_\alpha}^{p},
\]
we obtain
\[
C
\left(
\int_{\Omega}|v|^{1-\delta}f\,dx
\right)^{\frac{p}{1-\delta}}
\le
\|v\|_{X_\alpha}^{p},
\]
which is precisely \eqref{inequality2}.

\item[$(c)$] By part~$(a)$, we have
$
\Theta(\Omega)=\|V_\delta\|_{X_\alpha}^{p}.
$

Let \(v\in S_\delta\) satisfy
$
\|v\|_{X_\alpha}^{p}=\Theta(\Omega).
$

Since
\[
\int_\Omega |\,|v|\,|^{1-\delta}f\,dx
=
\int_\Omega |v|^{1-\delta}f\,dx
=1,
\]
we have \(|v|\in S_\delta\). Moreover,
$
\||v|\|_{X_\alpha}\le \|v\|_{X_\alpha}.
$
Hence,
$
\Theta(\Omega)
\le
\||v|\|_{X_\alpha}^{p}
\le
\|v\|_{X_\alpha}^{p}
=
\Theta(\Omega),
$
which implies
$
\||v|\|_{X_\alpha}^{p}=\Theta(\Omega).
$
Therefore, replacing \(v\) by \(|v|\) if necessary, we may assume that
$
v\ge0\text{ in }\Omega.
$
Now we define
\[
g:=
\left(
\int_\Omega
\left(\frac{v+V_\delta}{2}\right)^{1-\delta}
f\,dx
\right)^{\frac1{1-\delta}},
\]
and
\[
h:=
\frac{v+V_\delta}{2g}.
\]
Since \(0<1-\delta<1\), the function
$
t\mapsto t^{1-\delta}
$
is concave. Consequently,
\[
\left(\frac{a+b}{2}\right)^{1-\delta}
\ge
\frac{a^{1-\delta}+b^{1-\delta}}2,
\]
for all \(a,b\ge0\). Since \(v,V_\delta\in S_\delta\),
\[
g^{\,1-\delta}
=
\int_\Omega
\left(\frac{v+V_\delta}{2}\right)^{1-\delta}
f\,dx
\ge
\frac12
\int_\Omega
(v^{1-\delta}+V_\delta^{1-\delta})f\,dx
=1.
\]
Hence,
$
g\ge1.
$ Since \(h\in S_\delta\),
\[
\Theta(\Omega)
\le
\|h\|_{X_\alpha}^{p}
=
\frac1{g^{p}}
\left\|
\frac{v+V_\delta}{2}
\right\|_{X_\alpha}^{p}.
\]
Using the convexity of the norm,
\[
\left\|
\frac{v+V_\delta}{2}
\right\|_{X_\alpha}^{p}
\le
\frac{\|v\|_{X_\alpha}^{p}+\|V_\delta\|_{X_\alpha}^{p}}2
=
\Theta(\Omega),
\]
we obtain
\[
\Theta(\Omega)
\le
\frac{\Theta(\Omega)}{g^{p}}
\le
\Theta(\Omega),
\]
where the last inequality follows from \(g\ge1\).

Therefore,
$
g=1
$
and
\[
\left\|
\frac{v+V_\delta}{2}
\right\|_{X_\alpha}
=
\frac{\|v\|_{X_\alpha}+\|V_\delta\|_{X_\alpha}}{2}.
\]
Since \({X_\alpha}\) is uniformly convex, equality in the above relation is possible only if
$
v=V_\delta.
$
Finally, let \(\hat{h}\in X_\alpha\setminus\{0\}\) satisfy \eqref{sim}. We define
\[
\gamma
=
\left(
\int_\Omega |\hat{h}|^{1-\delta}f\,dx
\right)^{-\frac1{1-\delta}}.
\]
Then \(\gamma |\hat{h}|\in S_\delta\), and hence
$
\gamma |\hat{h}|=V_\delta.
$
Therefore,
\[
|\hat{h}|
=
\gamma^{-1}V_\delta
=
\gamma^{-1}\tau_\delta u_\delta.
\]
If \(\hat{h}\ge0\), then
$
\hat{h}=\gamma^{-1}\tau_\delta u_\delta,
$
while if \(\hat{h}\le0\), then
$
\hat{h}=-\gamma^{-1}\tau_\delta u_\delta.
$
Thus,
\[
\hat{h}=\pm\gamma^{-1}\tau_\delta u_\delta,
\]
which completes the proof.
\end{enumerate}
\end{proof}

\section*{Acknowledgement} Prashanta Garain acknowledges the financial support of the Anusandhan National Research Foundation (ANRF), India under the ARG-MATRICS grant, File No. ANRF/ARGM/2025\\/001315/MTR.


\begin{thebibliography}{10}

\bibitem{Annali}
Giovanni Anello, Francesca Faraci, and Antonio Iannizzotto.
\newblock On a problem of {H}uang concerning best constants in {S}obolev embeddings.
\newblock {\em Ann. Mat. Pura Appl. (4)}, 194(3):767--779, 2015.

\bibitem{AGmz}
Gurdev~Chand Anthal and Prashanta Garain.
\newblock Symmetry, existence and regularity results for a class of mixed local-nonlocal semilinear singular elliptic problem via variational characterization.
\newblock {\em Math. Z.}, 311(1):Paper No. 20, 2025.

\bibitem{ARad}
Rakesh Arora and Vicen\c tiu~D. R\u~adulescu.
\newblock Combined effects in mixed local-nonlocal stationary problems.
\newblock {\em Proc. Roy. Soc. Edinburgh Sect. A}, 155(1):10--56, 2025.

\bibitem{BDpur}
Kaushik Bal and Stuti Das.
\newblock Multiplicity of solutions for mixed local-nonlocal elliptic equations with singular nonlinearity, 2024.

\bibitem{BDp}
Kaushik Bal and Stuti Das.
\newblock Regularity results for a class of mixed local and nonlocal singular problems involving distance function, 2025.

\bibitem{BGmm22}
Kaushik Bal and Prashanta Garain.
\newblock Weighted and anisotropic {S}obolev inequality with extremal.
\newblock {\em Manuscripta Math.}, 168(1-2):101--117, 2022.

\bibitem{BGdie}
Kaushik Bal and Prashanta Garain.
\newblock Weighted anisotropic {S}obolev inequality with extremal and associated singular problems.
\newblock {\em Differential Integral Equations}, 36(1-2):59--92, 2023.

\bibitem{Ok24}
Linus Behn, Lars Diening, Jihoon Ok, and Julian Rolfes.
\newblock Nonlocal equations with degenerate weights, 2024.

\bibitem{BFKzamp}
M.~Belloni, V.~Ferone, and B.~Kawohl.
\newblock Isoperimetric inequalities, {W}ulff shape and related questions for strongly nonlinear elliptic operators.
\newblock {\em Z. Angew. Math. Phys.}, 54(5):771--783, 2003.
\newblock Special issue dedicated to Lawrence E. Payne.

\bibitem{BPugh}
A.~L. Bertozzi and M.~Pugh.
\newblock The lubrication approximation for thin viscous films: regularity and long-time behavior of weak solutions.
\newblock {\em Comm. Pure Appl. Math.}, 49(2):85--123, 1996.

\bibitem{Ghoshjga}
Souvik Bhowmick and Sekhar Ghosh.
\newblock Existence results for mixed local and nonlocal elliptic equations involving singularity and nonregular data.
\newblock {\em J. Geom. Anal.}, 35(8):Paper No. 218, 46, 2025.

\bibitem{Vecchi}
Stefano Biagi and Eugenio Vecchi.
\newblock Multiplicity of positive solutions for mixed local-nonlocal singular critical problems.
\newblock {\em Calc. Var. Partial Differential Equations}, 63(9):Paper No. 221, 45, 2024.

\bibitem{Biroud}
Kheireddine Biroud.
\newblock Mixed local and nonlocal equation with singular nonlinearity having variable exponent.
\newblock {\em J. Pseudo-Differ. Oper. Appl.}, 14(1):Paper No. 13, 24, 2023.

\bibitem{BGjga}
Sanjit Biswas and Prashanta Garain.
\newblock Existence {T}heory for a {C}lass of {S}emilinear {M}ixed {L}ocal and {N}onlocal {E}quations {I}nvolving {V}ariable {S}ingularities and {S}ingular {M}easures.
\newblock {\em J. Geom. Anal.}, 36(8):Paper No. 271, 2026.

\bibitem{BG26}
Sanjit Biswas and Prashanta Garain.
\newblock On the regularity theory for mixed local and nonlocal weighted quasilinear elliptic equations, 2026.

\bibitem{BGccm}
Sanjit Biswas and Prashanta Garain.
\newblock Regularity and existence for semilinear mixed local-nonlocal equations with variable singularities and measure data.
\newblock {\em Commun. Contemp. Math.}, 28(1):Paper No. 2550028, 39, 2026.

\bibitem{Boc-Or}
Lucio Boccardo and Luigi Orsina.
\newblock Semilinear elliptic equations with singular nonlinearities.
\newblock {\em Calc. Var. Partial Differential Equations}, 37(3-4):363--380, 2010.

\bibitem{Xiathesis}
Xia C.
\newblock {\em On a class of anisotropic problems}.
\newblock Dissertation zur Erlan-gung des Doktorgrades der Fakultät Mathematik und Physik der Albert-Ludwigs-Universität Freiburgim Breisgau, 2012.

\bibitem{Caninoetal}
Annamaria Canino, Luigi Montoro, Berardino Sciunzi, and Marco Squassina.
\newblock Nonlocal problems with singular nonlinearity.
\newblock {\em Bull. Sci. Math.}, 141(3):223--250, 2017.

\bibitem{Canino}
Annamaria Canino, Berardino Sciunzi, and Alessandro Trombetta.
\newblock Existence and uniqueness for {$p$}-{L}aplace equations involving singular nonlinearities.
\newblock {\em NoDEA Nonlinear Differential Equations Appl.}, 23(2):Art. 8, 18, 2016.

\bibitem{MB}
Philippe~G. Ciarlet.
\newblock {\em Linear and nonlinear functional analysis with applications}.
\newblock Society for Industrial and Applied Mathematics, Philadelphia, PA, 2013.

\bibitem{CRT}
M.~G. Crandall, P.~H. Rabinowitz, and L.~Tartar.
\newblock On a {D}irichlet problem with a singular nonlinearity.
\newblock {\em Comm. Partial Differential Equations}, 2(2):193--222, 1977.

\bibitem{Dama}
Lucio Damascelli.
\newblock Comparison theorems for some quasilinear degenerate elliptic operators and applications to symmetry and monotonicity results.
\newblock {\em Ann. Inst. H. Poincar\'e{} C Anal. Non Lin\'eaire}, 15(4):493--516, 1998.

\bibitem{DeCave}
Linda~Maria De~Cave.
\newblock Nonlinear elliptic equations with singular nonlinearities.
\newblock {\em Asymptot. Anal.}, 84(3-4):181--195, 2013.

\bibitem{Dna}
R.~Dhanya, Jacques Giacomoni, and Ritabrata Jana.
\newblock Interior and boundary regularity of mixed local nonlocal problem with singular data and its applications.
\newblock {\em Nonlinear Anal.}, 262:Paper No. 113940, 26, 2026.

\bibitem{Hitchhiker'sguide}
Eleonora Di~Nezza, Giampiero Palatucci, and Enrico Valdinoci.
\newblock Hitchhiker's guide to the fractional {S}obolev spaces.
\newblock {\em Bull. Sci. Math.}, 136(5):521--573, 2012.

\bibitem{DVap}
Serena Dipierro and Enrico Valdinoci.
\newblock Description of an ecological niche for a mixed local/nonlocal dispersal: an evolution equation and a new {N}eumann condition arising from the superposition of {B}rownian and {L}\'{e}vy processes.
\newblock {\em Phys. A}, 575:Paper No. 126052, 20, 2021.

\bibitem{Drabek}
Pavel Dr\'abek, Alois Kufner, and Francesco Nicolosi.
\newblock {\em Quasilinear elliptic equations with degenerations and singularities}, volume~5 of {\em De Gruyter Series in Nonlinear Analysis and Applications}.
\newblock Walter de Gruyter \& Co., Berlin, 1997.

\bibitem{Nmn}
G.~Ercole and G.~A. Pereira.
\newblock Fractional {S}obolev inequalities associated with singular problems.
\newblock {\em Math. Nachr.}, 291(11-12):1666--1685, 2018.

\bibitem{FW}
Csaba Farkas and Patrick Winkert.
\newblock An existence result for singular {F}insler double phase problems.
\newblock {\em J. Differential Equations}, 286:455--473, 2021.

\bibitem{Garainmn}
Prashanta Garain.
\newblock On a degenerate singular elliptic problem.
\newblock {\em Math. Nachr.}, 295(7):1354--1377, 2022.

\bibitem{Garainjga}
Prashanta Garain.
\newblock On a class of mixed local and nonlocal semilinear elliptic equation with singular nonlinearity.
\newblock {\em J. Geom. Anal.}, 33(7):Paper No. 212, 20, 2023.

\bibitem{Gjms}
Prashanta Garain.
\newblock Mixed anisotropic and nonlocal {S}obolev type inequalities with extremal.
\newblock {\em J. Math. Sci. (N.Y.)}, 281(5):633--645, 2024.

\bibitem{GarainHein}
Prashanta Garain.
\newblock Nonlocal singular problem and associated sobolev type inequality with extremal in the heisenberg group, 2025.

\bibitem{Gop2}
Prashanta Garain.
\newblock On mixed local-nonlocal {S}obolev-type inequalities and their connection with singular equations in the {H}eisenberg group.
\newblock {\em Opuscula Math.}, 46(3):305--321, 2026.

\bibitem{GKK}
Prashanta Garain, Wontae Kim, and Juha Kinnunen.
\newblock On the regularity theory for mixed anisotropic and nonlocal {$p$}-{L}aplace equations and its applications to singular problems.
\newblock {\em Forum Math.}, 36(3):697--715, 2024.

\bibitem{GMmed}
Prashanta Garain and Tuhina Mukherjee.
\newblock On a class of weighted {$p$}-{L}aplace equation with singular nonlinearity.
\newblock {\em Mediterr. J. Math.}, 17(4):Paper No. 110, 18, 2020.

\bibitem{Garaincpaa}
Prashanta Garain and Tuhina Mukherjee.
\newblock Quasilinear nonlocal elliptic problems with variable singular exponent.
\newblock {\em Commun. Pure Appl. Anal.}, 19(11):5059--5075, 2020.

\bibitem{GU21}
Prashanta Garain and Alexander Ukhlov.
\newblock Mixed local and nonlocal {S}obolev inequalities with extremal and associated quasilinear singular elliptic problems.
\newblock {\em Nonlinear Anal.}, 223:Paper No. 113022, 35, 2022.

\bibitem{GUaamp}
Prashanta Garain and Alexander Ukhlov.
\newblock Singular subelliptic equations and {S}obolev inequalities on {C}arnot groups.
\newblock {\em Anal. Math. Phys.}, 12(2):Paper No. 67, 18, 2022.

\bibitem{GSanona}
Jacques Giacomoni, Tuhina Mukherjee, and Konijeti Sreenadh.
\newblock Positive solutions of fractional elliptic equation with critical and singular nonlinearity.
\newblock {\em Adv. Nonlinear Anal.}, 6(3):327--354, 2017.

\bibitem{Gpur}
Abdelhamid Gouasmia.
\newblock Uniqueness results for mixed local and nonlocal equations with singular nonlinearities and source terms.
\newblock {\em Math. Nachr.}, 299(3):529--577, 2026.

\bibitem{Heinonen}
Juha Heinonen, Tero Kilpel\"ainen, and Olli Martio.
\newblock {\em Nonlinear potential theory of degenerate elliptic equations}.
\newblock Dover Publications, Inc., Mineola, NY, 2006.
\newblock Unabridged republication of the 1993 original.

\bibitem{HH}
Shuibo Huang and Hichem Hajaiej.
\newblock Lazer-{M}c{K}enna type problem involving mixed local and nonlocal elliptic operators.
\newblock {\em NoDEA Nonlinear Differential Equations Appl.}, 32(1):Paper No. 6, 45, 2025.

\bibitem{Stam}
David Kinderlehrer and Guido Stampacchia.
\newblock {\em An introduction to variational inequalities and their applications}, volume~88 of {\em Pure and Applied Mathematics}.
\newblock Academic Press, Inc. [Harcourt Brace Jovanovich, Publishers], New York-London, 1980.

\bibitem{Lz}
A.~C. Lazer and P.~J. McKenna.
\newblock On a singular nonlinear elliptic boundary-value problem.
\newblock {\em Proc. Amer. Math. Soc.}, 111(3):721--730, 1991.

\bibitem{LL}
Erik Lindgren and Peter Lindqvist.
\newblock Fractional eigenvalues.
\newblock {\em Calc. Var. Partial Differential Equations}, 49(1-2):795--826, 2014.

\bibitem{MV}
I.-I. Mezei and O.~Vas.
\newblock Existence results for some {D}irichlet problems involving {F}insler-{L}aplacian operator.
\newblock {\em Acta Math. Hungar.}, 157(1):39--53, 2019.

\end{thebibliography}
\end{document}